\RequirePackage[l2tabu, orthodox]{nag} 

\documentclass{amsart}

\pdfoutput=1

\usepackage[utf8]{inputenc} 
\usepackage[colorlinks, citecolor=blue, linkcolor=blue, breaklinks=true]{hyperref} 
\usepackage[stretch=10, shrink=10]{microtype} 
\usepackage{cases}
\usepackage{amscd, amssymb}  
\usepackage{commath} 
\usepackage{mathtools} 
\usepackage{doi}
\usepackage{color}
\usepackage[initials,msc-links,nobysame]{amsrefs}[2007/10/22]

\usepackage{paralist,enumitem}
\setlist{listparindent=0pt,parsep=3pt}
\setdefaultenum{1.}{}{}{}

\newcommand{\TitleWithUrl}[1]{\IfEmptyBibField{doi}%
  {\IfEmptyBibField{url}{\textit{#1}}%
    {\IfEmptyBibField{eprint}{\href {\BibField{url}}{\textit{#1}}}{\textit{#1}}}%
    }%
  {\href {https://doi.org/\BibField{doi}}{\textit{#1}}}}
\renewcommand{\eprint}[1]{\IfEmptyBibField{url}{\url{#1}}%
  {\href {\BibField{url}}{#1}}}

\BibSpec{article}{%
    +{}  {\PrintAuthors}                {author}
    +{,} { \TitleWithUrl}               {title}
    +{.} { }                            {part}
    +{:} { \textit}                     {subtitle}
    +{,} { \PrintContributions}         {contribution}
    +{.} { \PrintPartials}              {partial}
    +{,} { }                            {journal}
    +{}  { \textbf}                     {volume}
    +{}  { \PrintDatePV}                {date}
    +{,} { \issuetext}                  {number}
    +{,} { \eprintpages}                {pages}
    +{,} { }                            {status}
    +{,} { available at \eprint}        {eprint}
    +{}  { \parenthesize}               {language}
    +{}  { \PrintTranslation}           {translation}
    +{;} { \PrintReprint}               {reprint}
    +{.} { }                            {note}
    +{.} {}                             {transition}
    +{}  {\SentenceSpace \PrintReviews} {review}
}

\BibSpec{collection.article}{%
    +{}  {\PrintAuthors}                {author}
    +{,} { \TitleWithUrl}                     {title}
    +{.} { }                            {part}
    +{:} { \textit}                     {subtitle}
    +{,} { \PrintContributions}         {contribution}
    +{,} { \PrintConference}            {conference}
    +{}  {\PrintBook}                   {book}
    +{,} { }                            {booktitle}
    +{,} { \PrintDateB}                 {date}
    +{,} { pp.~}                        {pages}
    +{,} { }                            {status}
    +{,} { available at \eprint}        {eprint}
    +{}  { \parenthesize}               {language}
    +{}  { \PrintTranslation}           {translation}
    +{;} { \PrintReprint}               {reprint}
    +{.} { }                            {note}
    +{.} {}                             {transition}
    +{}  {\SentenceSpace \PrintReviews} {review}
}
\BibSpec{book}{%
    +{}  {\PrintPrimary}                {transition}
    +{,} { \TitleWithUrl}               {title}
    +{.} { }                            {part}
    +{:} { \textit}                     {subtitle}
    +{,} { \PrintEdition}               {edition}
    +{}  { \PrintEditorsB}              {editor}
    +{,} { \PrintTranslatorsC}          {translator}
    +{,} { \PrintContributions}         {contribution}
    +{,} { }                            {series}
    +{,} { \voltext}                    {volume}
    +{,} { }                            {publisher}
    +{,} { }                            {organization}
    +{,} { }                            {address} 
    +{,} { \PrintDateB}                 {date}
    +{,} { }                            {status}
    +{}  { \parenthesize}               {language}
    +{}  { \PrintTranslation}           {translation}
    +{;} { \PrintReprint}               {reprint}
    +{.} { }                            {note}
    +{.} {}                             {transition}
    +{}  {\SentenceSpace \PrintReviews} {review}
}

\BibSpec{thesis}{%
    +{}  {\PrintAuthors}                {author}
    +{.} { \PrintDate}                  {date}
    +{.} { \TitleWithUrl}               {title}
    +{:} { \textit}                     {subtitle}
    +{,} { \PrintThesisType}            {type}
    +{,} { }                            {organization}
    +{,} { }                            {address}
    +{,} { \eprint}                     {eprint}
    +{,} { }                            {status}
    +{}  { \parenthesize}               {language}
    +{}  { \PrintTranslation}           {translation}
    +{;} { \PrintReprint}               {reprint}
    +{.} { }                            {note}
    +{.} {}                             {transition}
    +{}  {\SentenceSpace \PrintReviews} {review}
}

\newtheorem{theorem}{Theorem}[section]
\newtheorem{lemma}[theorem]{Lemma}

\newtheorem{proposition}[theorem]{Proposition}

\theoremstyle{definition}
\newtheorem{definition}[theorem]{Definition}

\theoremstyle{remark}
\newtheorem{remark}[theorem]{Remark}

\numberwithin{equation}{section}

\newcommand{\cratio}{\mathrm{cr}}

\title{Infinitesimal Darboux transformation and semi-discrete mKdV equation}

\author{Joseph Cho}
\address[Joseph Cho]{Institute of Discrete Mathematics and Geometry, TU Wien, Wiedner Hauptstrasse 8-10/104, 1040 Wien, Austria}
\email{jcho@geometrie.tuwien.ac.at}
 
\author{Wayne Rossman}
\address[Wayne Rossman]{Department of Mathematics, Graduate School of Science, Kobe University, 1-1 Rokkodai-cho, Nada-ku, Kobe 657-8501, Japan}
\email{wayne@math.kobe-u.ac.jp}

\author{Tomoya Seno}
\address[Tomoya Seno]{Department of Mathematics, Graduate School of Science, Kobe University, 1-1 Rokkodai-cho, Nada-ku, Kobe 657-8501, Japan}
\email{tseno@math.kobe-u.ac.jp}

\date{}
\makeatletter
\@namedef{subjclassname@2020}{\textup{2020} Mathematics
  Subject Classification}
\makeatother

\begin{document}

\begin{abstract}
	We connect certain continuous motions of discrete planar curves resulting in semi-discrete potential mKdV equation with Darboux transformations of smooth planar curves.
	In doing so, we define infinitesimal Darboux transformations that include the aforementioned motions, and also give an alternate geometric interpretation for establishing the semi-discrete potential mKdV equation.
\end{abstract}

\maketitle

\section{Introduction}
We aim to establish an equivalence between two different approaches to discretizations, of distinctly different objects, with integrability.
At the heart of discrete differential geometry preserving integrability lies the field of integrable systems, a field that historically stems from the surface theory of differential geometry: Classical geometers sought methods to obtain a new surface from a given surface keeping certain properties, oftentimes represented by a particular non-linear partial differential equation (PDE), giving birth to transformation theory.

In such pursuit, one generally needs to solve a system of PDEs; however, there is a general principle which allows one to obtain new surfaces algebraically, known as the permutability or the superposition principle.
Discrete differential geometry seeks to recover the fully discrete theory and the corresponding nonlinear difference-difference equations with such integrability intact: often this is achieved by focusing on either
	\begin{enumerate}
		\item the compatibility condition of the moving frames, or
		\item the permutability of the associated transformations.
	\end{enumerate}
The semi-discrete theory, concerned with differential-difference equations described by a discrete parameter and a smooth parameter, can be recovered analogously, where now the transformation itself describes the semi-discrete structure.
Starting with the work \cite{levi_backlund_1980}, some of the earlier works that follow such approach include \cites{konopelchenko_elementary_1982, quispel_linear_1984, nijhoff_backlund_1984, nimmo_superposition_1997, doliwa_transformations_2000, schief_isothermic_2001, bobenko_discrete_2008, levi_nonlinear_1981}.
Since both approaches are aimed at preserving integrable structures, it is natural to expect a fair degree of equivalence between the two approaches, resulting in the same discretization.

A prototypical example of such equivalence in surface theory can be found in the class of $K$-surfaces, those surfaces with constant (negative) Gaussian curvature.
In studying this class of surfaces, Bour \cite{bour_theorie_1862} identified the now celebrated sine-Gordon equation via the compatibility condition.
It was then the work of Bäcklund \cite{backlund_om_1883} which used the concept of tangential line congruences to obtain new $K$-surfaces from a given one, a process commonly referred to as the Bäcklund transformation.
Finally, Bianchi \cite{bianchi_sulla_1892} showed that there exists a superposition principle for the Bäcklund transformation, commonly referred to as Bianchi permutability.
With these properties of $K$-surfaces, discrete $K$-surfaces were first studied in \cites{sauer_parallelogrammgitter_1950, wunderlich_zur_1951} by considering the Bianchi quadrilateral of the Bäcklund transformations of a smooth $K$-surface; on a separate note, the work \cite{hirota_nonlinear_1977} considered the discrete integrable analogue of the sine-Gordon equation.
Then it was the work of \cite{bobenko_discrete_1996} that clarified the relationship between the two, linking the two approaches to discrete theory.
Furthermore, the result \cite{inoguchi_discrete_2014} used the fact that certain isoperimetric deformation of space curves map out a $K$-surface, and showed that their approach to discrete moving frames results in the same discretization as that of $K$-surfaces.

Here we will establish another equivalence between the two approaches, albeit each discretizing a different object.

On one hand, \emph{isothermic surfaces}, first examined by Bour \cite{bour_theorie_1862} and studied by classical geometers such as Bianchi \cites{bianchi_complementi_1905, bianchi_ricerche_1904}, Calapso \cite{calapso_sulla_1903}, and Darboux \cite{darboux_sur_1899}, constitute an integrable class of surfaces with its own \emph{Darboux transformation} and Bianchi permutability; such integrable nature has been revisited in modern times (see, for example, \cites{burstall_isothermic_2006, burstall_isothermic_2011, burstall_curved_1997, hertrich-jeromin_remarks_1997, hertrich-jeromin_mobius_2001, cieslinski_isothermic_1995}).
In fact, the permutability of Darboux transformations leads to \emph{discrete isothermic surfaces} as defined in \cite{bobenko_discrete_1996-1} via a certain cross-ratios condition.
Building on this, the work \cite{muller_semi-discrete_2013} examined the class of \emph{semi-discrete isothermic surfaces}, where the surface is parametrized by a smooth parameter and a discrete parameter, recovering various analogous theory including the concept of Christoffel duals.

The connection between transformation theory and semi-discrete isothermic surfaces has been identified in \cite{burstall_semi-discrete_2016}.
By prescribing space curves with a polarization, this work defined a \emph{Darboux pair} of curves as a pair of curves enveloping a circle congruence preserving the polarization.
After showing such a Darboux transformation of a polarized curve enjoys a certain cross-ratio condition, they characterized semi-discrete isothermic surfaces as iterations of Darboux transformations of polarized curves.
Hence, the smooth parameter of a semi-discrete isothermic surface describes the curves, while the discrete parameter describes the Darboux transformation.

In fact, holomorphic functions can be viewed as isothermic surfaces whose images are in the plane, identified with the complex plane, with the prototypical application being the Weierstrass representation of minimal surfaces \cite{weierstrass_untersuchungen_1866}.
Discrete holomorphic functions \cite{bobenko_discrete_1996-1} and \emph{semi-discrete holomorphic functions} have been defined analogously, paving the way to defining Weierstrass representation for discrete and semi-discrete minimal surfaces \cites{muller_semi-discrete_2013, bobenko_discrete_1996-1}.
Therefore, one can view semi-discrete holomorphic functions as successive Darboux transformations of polarized plane curves.

On the other hand, works such as \cites{lamb_solitons_1976, goldstein_korteweg-vries_1991} explored the modified Korteweg--de Vries (mKdV) equation arising as the compatibility condition of the moving frames describing certain motions of space curves.
Then it was the result of \cites{matsuura_discrete_2012, matsuura_discrete_2012-1, inoguchi_explicit_2012, inoguchi_motion_2012} that revealed a relationship between various semi-discrete and discrete analogues of the KdV and mKdV equations (see, for example, \cites{hirota_nonlinear_1977-1, hirota_discretization_1998}), and the compatibility condition associated with discretization of the moving frames.
In fact, by considering certain motions of discrete space curves and the corresponding moving frames, \cite{kaji_linkage_2019} recovered a particular form of the \emph{semi-discrete potential mKdV equation} that can be found in \cite{wadati_backlund_1974} in the context of the transformation and permutability of solutions to the smooth mKdV equation, suggesting a close relationship between the moving frames approach and transformation theory.

It is this relationship that we identify in this paper.

In Section \ref{sect:Darboux}, we focus on the semi-discrete isothermic surfaces whose images are in the plane, i.e., semi-discrete holomorphic functions.
First, we consider the Darboux transformations of smooth polarized curves, adapting the results in \cite{burstall_semi-discrete_2016} for plane curves where the complex number system replaces the Clifford algebra; in addition, we identify the Darboux transformations that preserve the arc-length polarization, also known as tractrix construction or bicycle correspondence (see, for example, \cites{bor_tire_2020, tabachnikov_bicycle_2017, hoffmann_discrete_2008}).
Then, we define \emph{infinitesimal Darboux transformation} of discrete polarized curves in Definition \ref{def:dFlow}, and examine when the infinitesimal transformation preserves the discrete arc-length polarization (see Definition \ref{def:discArcPol}) in Proposition \ref{prop:discArc}.

Such examination is crucial for connecting the infinitesimal Darboux transformation to the aforementioned continuous motion of curves, which is the subject of our attention in Section \ref{sect:isoperimetric}.
After briefly reviewing the material, we precisely define the potential function used in the semi-discrete potential mKdV equation in terms of the geometric data of the continuous motion.
Such precise geometric definition of the potential function allows us to identify the continuous motion yielding the semi-discrete potential mKdV equation as an infinitesimal Darboux transformation preserving discrete arc-length polarization in Proposition \ref{prop:isoDarboux}, allowing us to obtain various characterizations of the smooth motion of discrete plane curves resulting in semi-discrete potential mKdV equation, in Theorem \ref{thm:semidiscrete}.
Finally, as an application, we introduce an efficient approach to obtaining the semi-discrete potential mKdV equation without involving any frames, directly from the infinitesimal Darboux transformation equation.

\section{Infinitesimal Darboux transformation of discrete curves}\label{sect:Darboux}

Let $\mathbb{R}^2$ be the Euclidean plane with the standard inner product denoted by $\cdot$, and denote the length of a vector by $| \:\: |$.
Throughout the paper, we identify $\mathbb{R}^2 \cong \mathbb{C}$ via $(x,y) \sim z = x + i y$, and $\bar{z}$ is the usual complex conjugation of $z$. 
 
\subsection{Darboux transformation}
Darboux transformations of curves in $\mathbb{R}^n$ polarized by a (non-vanishing) quadratic differential without poles are defined in \cite{burstall_semi-discrete_2016} (for those with poles, see \cite{fuchs_transformations_2019}).
In this paper, we treat the special case when the curve is in $\mathbb{R}^2$, adapting the results of \cite{burstall_semi-discrete_2016} in terms of complex multiplication in lieu of Clifford multiplication.

For this, let $I \subset \mathbb{R}$ be an interval parametrized by $s$.
As in \cite[Definition 2.1]{burstall_semi-discrete_2016}, two curves $x, \hat{x}: I \to \mathbb{C}$ are said to be a \emph{Ribaucour pair} if they envelop a common circle congruence, that is, $x(s)$ and $\hat{x}(s)$ are both tangent to a circle $c(s)$.
Recalling from \cite[Definition and Lemma 2.2]{burstall_semi-discrete_2016} (see also \cite{muller_semi-discrete_2017}) that the \emph{tangential cross ratio} $\cratio : I \to \mathbb{C}$ of the two curves $x,\hat{x}:I \to \mathbb{C}$ is defined by
	\[
		\cratio := \frac{x'\hat{x}'}{(x-\hat{x})^2},
	\]
where $' = \frac{\dif{}}{\dif{s}}$, we have that the two curves are a Ribaucour pair if and only if the tangential cross ratio is real everywhere.

To consider Darboux transformations of curves as in \cite{burstall_semi-discrete_2016}, we consider curves \emph{polarized} by a non-vanishing quadratic differential $\frac{\dif{s}^2}{m}$ with $m : I \to \mathbb{R}^{\times}$.

\begin{remark}
	\emph{Polarizations} are quadratic differentials introduced for a coordinate-free approach to surface classes characterized by the existence of special coordinate systems such as the $L$-isothermic surfaces \cite[Definition 2.1]{musso_bianchi-darboux_2000} or isothermic surfaces \citelist{\cite{bernstein_non-special_2001}*{Definition 2.16} \cite{hertrich-jeromin_mobius_2001}*{p.\ 190} \cite{bobenko_painleve_2000}*{Lemma 2.3.2}}.
	Following this principle, polarized curves were introduced in \cite{burstall_semi-discrete_2016} to define Darboux transformations independently of parametrizations.
\end{remark}

\begin{definition}[{\cite[Definition 2.4]{burstall_semi-discrete_2016}}]
	For some constant $\mu \in \mathbb{R}^\times$, two polarized curves $x,\hat{x} : (I,\frac{\dif{s}^2}{m}) \to \mathbb{C}$ are called a \emph{Darboux pair with parameter $\mu$} if
		\begin{equation}\label{cr}
			\cratio \dif{s}^2 = \frac{x'\hat{x}'}{(x-\hat{x})^2} \dif{s}^2 = \frac{\mu}{m} \dif{s}^2.
		\end{equation}
	Either curve of a Darboux pair is called a \emph{Darboux transform} of the other curve.
\end{definition}

\begin{remark}
The definition of Darboux transformations allows us to check the following facts:
	\begin{itemize}
		\item Any Darboux pair is a Ribaucour pair, i.e., they envelop a common circle congruence.
		\item Darboux transformations are independent of the choice of parametrization of the curve.
		\item Darboux transformations are dependent on the choice of polarization (see Figure \ref{fig:dT}).
	\end{itemize}
\end{remark}

\begin{figure}
	\includegraphics{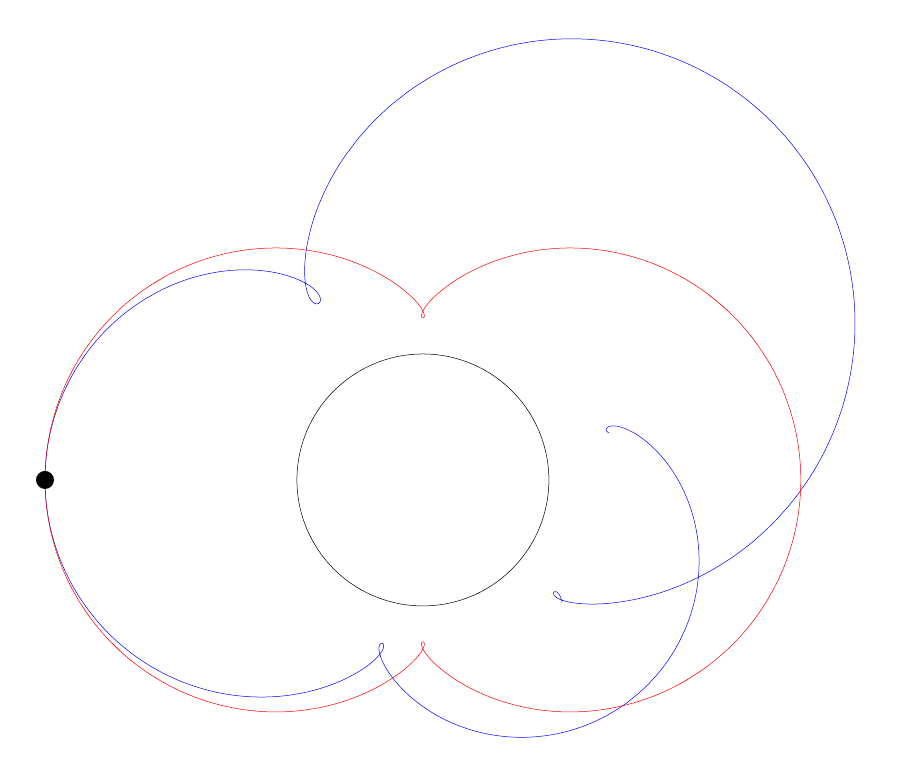}
	\caption{Two Darboux transformations (colored in red and blue) of a circle (colored in black) with identical parameter $\mu$ and initial condition (marked by a point), but with different polarizations.}
	\label{fig:dT}
\end{figure}

By \eqref{cr}, given a polarized curve $x$, a Darboux transform $\hat{x}$ of $x$ can be found by solving the Riccati-type equation
	\[
		\hat{x}' = \frac{\mu}{m}(x-\hat{x})(x')^{-1}(x-\hat{x}),
	\] 
determined by the choice of parameter $\mu$ and an initial condition, giving us a $3$-parameter family of Darboux transforms $\hat{x}$.


Let $x, \hat{x}$ be a Darboux pair with parameter $\mu$; then we can generally find $y := x+\frac{1}{r}x' = \hat{x}+\frac{1}{\hat{r}}\hat{x}'  \in \mathbb{C}$ for some $r, \hat{r} : I \to \mathbb{R}$.
Under such setting, we prove the following two lemmata that will be useful later.

\begin{lemma}[{\cite[Equation (2.2)]{burstall_semi-discrete_2016}}]\label{lem:rlemma1}
	If $x,\hat{x}$ is a Darboux pair, then $r$ and $\hat{r}$ are related by
		\[
			\frac{\hat{r}}{r}=-\frac{|x-\hat{x} |^2}{|x'|^2}\frac{\mu}{m}.
		\]
\end{lemma}
\begin{proof}
	If $(x,\hat{x})$ are a Darboux pair, then they envelop a common circle congruence; hence, the distance between $y$ and the corresponding points $x$ and $\hat{x}$ are the same, that is, $|x-y|^2=|\hat{x}-y|^2$.
	From this, we obtain by \eqref{cr},
		\[
			\frac{\hat{r}^2}{r^2}
				= \left | \frac{\hat{x}'^2}{x'^2} \right |
				= \frac{|x-\hat{x}|^4}{|x'|^4} \frac{\mu^2}{m^2}.
		\]
	Now, without loss of generality, assume $x$ and $\hat{x}$ are symmetric about the real-axis, at the point under consideration, so that $(x-\hat{x})^2$ is always negative while $\frac{\hat{r}}{r}$ and $x'\hat{x}'$ have the same sign.
	Thus $\frac{\hat{r}}{r}$ and the cross ratio $\frac{x' \hat{x}'}{(x-\hat{x}')^2}=\frac{\mu}{m}$ have different signs, giving us the desired conclusion.
\end{proof}

\begin{lemma}\label{lem:rlemma2}
	If $x,\hat{x}$ is a Darboux pair, then $r$ and $\hat{r}$ satisfy
		\[
			\frac{r}{2} =\frac{(\hat{x}-x) \cdot x'}{| \hat{x}-x |^2}\quad\text{and}\quad
			\frac{\hat{r}}{2} =\frac{(x - \hat{x}) \cdot \hat{x}'}{| \hat{x}-x |^2}.
		\]
\end{lemma}
\begin{proof}
	We first note that
		\begin{align*}
			| x-y |^2 & = \left| x-x-\tfrac{1}{r}x' \right|^2=\tfrac{1}{r^2} | x' |^2, \\
			| \hat{x}-y |^2 & = \left| \hat{x}-x-\tfrac{1}{r}x' \right|^2
				= | \hat{x}-x |^2 - \tfrac{2}{r} (\hat{x}-x) \cdot x'+\tfrac{1}{r^2}|x'|^2.
		 \end{align*}
	 Hence, the fact that $| x-y |^2 = | \hat{x}-y |^2$ gives us one of the statements.
	 The other statement is shown similarly.
\end{proof}

\subsection{Arc-length polarization}
We now consider arc-length polarized curves and their transformations.
\begin{definition}[{\cite[p.\ 48]{burstall_semi-discrete_2016}}]
	For a polarized curve $x:(I,\frac{\dif{s}^2}{m}) \to \mathbb{C}$, the differential $\frac{\dif{s}^2}{m}$ is called the \emph{arc-length polarization} of $x$ if
		\[
			\frac{\dif{s}^2}{m} = | {\dif{x}} |^2.
		\]
\end{definition}

\begin{remark}\label{rem:arclength}
	If $x:(I,\frac{\dif{s}^2}{m}) \to \mathbb{C}$ is an arc-length polarized curve, then $s$ is an arc-length parameter if and only if $m \equiv 1$.
\end{remark}

Now we would like to see when the arc-length polarization is preserved under Darboux transformations, hence becoming bicycle correspondence (see, for example, \cites{burstall_semi-discrete_2016, tabachnikov_bicycle_2017}).
First, introducing $\Lambda:=|\hat{x}-x|^2$, we have via Lemmata \ref{lem:rlemma1} and \ref{lem:rlemma2} that
	\begin{equation}\label{eqn:lambda}
 		 \Lambda' = 2(\hat{x}-x) \cdot (\hat{x}'-x')
			= ( -\hat{r}- r) \Lambda
			= r\left(\frac{\mu}{m |x'|^2} \Lambda -1\right)\Lambda.
	\end{equation}
Now, if $\frac{\dif{s}^2}{m}$ is an arc-length polarization of $x$ so that $\frac{1}{m} \equiv |x'|^2$, then uniqueness of solutions to ordinary differential equations tells us that if $\Lambda = \frac{1}{\mu}$ at one point $s_0 \in I$, then $\Lambda \equiv \frac{1}{\mu}$, a condition that is equivalent to $\frac{\dif{s}^2}{m} = |\hat{x}'|^2 \dif{s}^2 = |{\dif{\hat{x}}}|^2$ by \eqref{cr}.
Therefore, we have shown the following:
\begin{proposition}\label{prop:arclength}
	Let $x,\hat{x}:(I,\frac{dx^2}{m}) \to \mathbb{C}$ be a Darboux pair with parameter $\mu$, where $\frac{\dif{s}^2}{m}$ is the arc-length polarization of $x$.
	Then $\frac{\dif{s}^2}{m}$ is also the arc-length polarization of $\hat{x}$ if and only if $|\hat{x}-x|^2=\tfrac{1}{\mu}$ at one point $s_0 \in I$.
\end{proposition}
 
Thus there is a $2$-parameter family of possible Darboux transforms $\hat{x}$ keeping arc-length condition, given by choosing $\mu$ and the initial condition at the correct distance.
 
\subsection{Infinitesimal Darboux transformation of discrete polarized curves}

Viewing semi-discrete holomorphic functions as semi-discrete isothermic surfaces whose image is contained in a plane, the work of \cite{burstall_semi-discrete_2016} shows that semi-discrete holomorphic functions can be characterized as successive Darboux transformations of a smooth polarized curve, i.e.\ the smooth parameter of a semi-discrete holomorphic function parametrizes the curve while the discrete parameter represents the transformation.
In this section, we switch the roles of the smooth and discrete parameters: We let the discrete parameter of a semi-discrete holomorphic function represent the parameter of a discrete curve, and smooth parameter represent the continuous transformation, allowing us to consider semi-discrete holomorphic functions as the images of \emph{infinitesimal Darboux transformations}, or \emph{Darboux deformations}, of discrete polarized curves.

For this, let $\Sigma \subset \mathbb{Z}$ be a discrete interval (see, for example, \cite[\S 2.3]{burstall_discrete_2020}), and let $\mu$ be a strictly positive or negative function on (unoriented) edges of $\Sigma$.
We call $(\Sigma,\tfrac{1}{\mu})$ a \emph{discrete polarized domain}.
A \emph{discrete polarized curve} is a discrete curve $x:(\Sigma,\tfrac{1}{\mu}) \to \mathbb{C}$ whose polarization is given by some such function $\mu$.

\begin{definition}\label{def:dFlow}
	Let $x^0 : (\Sigma, \frac{1}{\mu}) \to \mathbb{C}$ be a discrete polarized curve, and let $m : I \to \mathbb{R}^\times$.
	An \emph{infinitesimal Darboux transformation with parameter function $m$} is a smooth motion $x : \Sigma \times I \to \mathbb{C}$ of the discrete curve $x^0$, i.e.\ $x_n(s_0) = x^0_n$ for some $s_0 \in I$, so that on every edge $(n,n+1)$,
		\begin{equation}\label{eqn:dFlow}
			\frac{x_n' x_{n+1}'}{(x_n - x_{n+1})^2} \dif{s}^2 = \frac{\mu_{(n,n+1)}}{m} \dif{s}^2.
		\end{equation}
\end{definition}
Given a discrete polarized curve $x^0 : (\Sigma, \frac{1}{\mu}) \to \mathbb{C}$ and a function $m : I \to \mathbb{R}^\times$, one can obtain an infinitesimal Darboux transformation with parameter function $m$ via a given initial condition curve $x_{n_0}(s) : I \to \mathbb{C}$ polarized by the quadratic differential $\frac{\dif{s}^2}{m}$.
Comparing to the usual Darboux transformations, this initial polarized curve represents a function worth of spectral parameters and initial conditions.

With discrete polarized curves defined, one can also consider the arc-length polarization of a discrete curve:
\begin{definition}\label{def:discArcPol}
For a discrete polarized curve $x : (\Sigma, \frac{1}{\mu}) \to \mathbb{C}$, $\mu$ is called the \emph{(discrete) arc-length polarization} of $x$ if, on every edge $(n,n+1)$,
	\begin{equation}
		\frac{1}{\mu_{(n,n+1)}} = |x_n - x_{n+1}|^2.
	\end{equation}
\end{definition}
\begin{remark}
	As in the smooth case (see Remark \ref{rem:arclength}), a discrete polarized curve with arc-length polarization is \emph{arc-length parametrized} as in \cite[Definition 2.4]{hoffmann_discrete_2009} (see also \cites{hoffmann_discrete_2000, hoffmann_discrete_2004}) if and only if $\mu \equiv 1$.
\end{remark}
Now we identify the condition for infinitesimal Darboux transformations keeping the arc-length polarization.
\begin{proposition}\label{prop:discArc}
	For a discrete polarized curve $x^0 : (\Sigma, \frac{1}{\mu}) \to \mathbb{C}$ with arc-length polarization $\mu$, let $x$ be an infinitesimal Darboux transformation of $x^0 = x(0)$ with parameter function $m$.
	Then $\mu$ will be the (discrete) arc-length polarization of $x_n(s)$ for any $s \in I$ if and only if $\frac{ds^2}{m}$ is the (smooth) arc-length polarization of the initial condition  curve $x_{n_0}(s)$ for some $n_0 \in \Sigma$.
\end{proposition}
\begin{proof}
	The definition of infinitesimal Darboux transformation \eqref{eqn:dFlow} tells us the pair of smooth curves $x_n(s)$ and $x_{n+1}(s)$ are a Darboux pair with parameter $\mu_{(n,n+1)}$ on any edge $(n,n+1)$, where $\frac{ds^2}{m}$ polarizes both curves.
	First, assume that $\mu$ is the (discrete) arc-length polarization of $x_n(s)$ for any $s \in I$, so that
		\[
			\frac{1}{\mu_{(n,n+1)}} = |x_n - x_{n+1}|^2 = \Lambda
		\]
	is constant in $s$.
	Then \eqref{eqn:lambda} tells us that $\frac{1}{m} = |x_n'|^2$, and applying Proposition~\ref{prop:arclength} gives us one direction.
	The other direction is a direct consequence of Proposition~\ref{prop:arclength}.
\end{proof}

%

\section{Certain isoperimetric deformations as infinitesimal Darboux transformations}\label{sect:isoperimetric}

\subsection{Preliminaries}
We briefly review the continuous motions of discrete curves resulting in the potential mKdV equations as considered in \cites{kaji_linkage_2019}, adapted for the case of discrete plane curves, and with added minor generalization (see also \cites{doliwa_elementary_1994, doliwa_integrable_1995, hisakado_motion_1995, matsuura_discrete_2012, inoguchi_explicit_2012}).
Let $x : \Sigma \rightarrow \mathbb{C}$ be a planar curve which is regular, that is, $\det{(x_n - x_{n-1} \ \  x_{n+1} - x_n)} \neq 0$ on any three consecutive vertices.
The unit tangent vector $T_n$ and the unit normal vector $N_n$ at $x_n$ can be defined by
	\[
		T_n := \frac{x_{n+1} - x_n}{a_n}, \quad N_n :=R(\tfrac{\pi}{2}) T_n,
	\]
where $a_n:=| x_{n+1}-x_n |$, and $R(\theta)$ denotes counterclockwise rotation about the origin by $\theta$.
Then a notion of \emph{discrete curvature $\kappa_n$} at $x_n$ can be defined via the Frenet-type equation (see, for example, \cite{matsuura_discrete_2012-1})
	\begin{equation}\label{eqn:Frenet}
		F_{n+1} = F_n R(\kappa_{n+1}) =: F_n L_n,
	\end{equation}
where $F_n$ is the frame $F_n=(T_n \ N_n) \in \mathrm{SO(2)}$.
Defining $\psi_n$ as the argument of $T_n$, $\kappa$ and $\psi$ are related by 
	\begin{equation}\label{eqn:psi1}
		\psi_{n+1}=\psi_n+\kappa_{n+1}.
	\end{equation}
	
Now consider a continuous motion $x_n(s):\Sigma \times I \rightarrow \mathbb{C}$ of a discrete curve $x_n(0) : \Sigma \to \mathbb{C}$ with constant speed of motion, that is, $| x_n' |$ is constant for all $(n, s) \in \Sigma \times I$.
By an appropriate constant scaling of the smooth parameter $s$, we may assume without loss of generality that $| x_n' | \equiv 1$, allowing for the definition of $w_n$ via 
	\[
		x_n'=\cos{w_n} T_n+\sin{w_n} N_n = F_n \begin{pmatrix} \cos w_n \\ \sin w_n \end{pmatrix}.
	\]
Since we have that $a_n^2 = |x_{n+1} - x_n|^2 = (x_{n+1} - x_n) \cdot (x_{n+1} - x_n)$,
	\begin{align*}
		a_n' &= (x_{n+1} - x_n)' \cdot T_n = \left(F_{n+1}\begin{pmatrix} \cos w_{n+1} \\ \sin w_{n+1} \end{pmatrix} - F_{n}\begin{pmatrix} \cos w_n \\ \sin w_n \end{pmatrix} \right)\cdot T_n \\
			&= F_n \left(L_n \begin{pmatrix} \cos w_{n+1} \\ \sin w_{n+1} \end{pmatrix} - \begin{pmatrix} \cos w_n \\ \sin w_n \end{pmatrix}\right) \cdot F_n \begin{pmatrix} 1 \\ 0 \end{pmatrix} \\
			&= \cos{(w_{n+1}+\kappa_{n+1})}-\cos{w_n}
	\end{align*}
where we used the Frenet equation \eqref{eqn:Frenet}.
Thus, excluding trivial deformations, we have that the motion is \emph{isoperimetric}, i.e.\ $a_n' \equiv 0$, if and only if
	\begin{equation}\label{8}
		w_{n+1}+\kappa_{n+1}= - w_n.
	\end{equation}
\begin{remark}
	The continuous motion considered in \cite{kaji_linkage_2019} assumed that $a_n(0)$ is constant for all $n \in \Sigma$, and defined the motion via
		\[
			x_n'=\frac{a_n}{\rho}(\cos{w_n} T_n+\sin{w_n} N_n)
		\]
	for some constant $\rho > 0$, so that $|x_n'| \equiv a_n/ \rho$ is constant.
	On contrast, we do not a priori require that $a_n(0)$ be constant for all $n \in \Sigma$, but only that $|x_n'|$ is constant.
\end{remark}
Hence, assuming that the continuous motion is isoperimetric, we have that $F'_n$ satisfies
	\begin{equation}\label{7}
		F'_n=F_n M_n, \
		M_n:=\frac{1}{a_n}
			\begin{pmatrix}
				0 & 2\sin{w_n} \\
				-2\sin{w_n} & 0
			\end{pmatrix}.
	\end{equation} 

Defining a potential function $\theta_n$ via the relations
	\begin{equation}\label{eqn:theta}
		\psi_n=\frac{\theta_{n+1}+\theta_n}{2}, \: w_n=-\frac{\theta_{n+1}-\theta_n}{2}, \:\text{and } \kappa_n = \frac{\theta_{n+1} - \theta_{n-1}}{2},
	\end{equation}
one can calculate using the compatibility condition
	\[
		L_n' = L_n M_{n+1} - M_n L_n
	\]
that
	\begin{equation}\label{eqn:mKdVrel}
		\begin{multlined}
			\left(\frac{\theta_{n+2}+\theta_{n+1}}{2}\right)' - \left(\frac{\theta_{n+1}+\theta_n}{2}\right)' \\
				\qquad = \frac{2}{a_{n+1}}\sin{\left(\frac{\theta_{n+2}-\theta_{n+1}}{2}\right)} - \frac{2}{a_n}\sin{\left(\frac{\theta_{n+1}-\theta_{n}}{2}\right)}.
		\end{multlined}
	\end{equation}
Therefore, we can infer that $\theta_n$ satisfies the following semi-discrete potential mKdV equation  (cf.\ \citelist{\cite{wadati_backlund_1974}*{Equation (7)} \cite{kaji_linkage_2019}*{Equation (4.21)}}):
	\begin{equation}\label{mKdV}
		\left(\frac{\theta_{n+1}+\theta_n}{2}\right)' = \frac{2}{a_n}\sin{\left(\frac{\theta_{n+1}-\theta_n}{2}\right)}.
	\end{equation}

\subsection{Geometric interpretation of the semi-discrete potential mKdV equation}
The relational definition of $\theta_n$ as in \eqref{eqn:theta} results in a difference equation for the semi-discrete potential mKdV equation as in \eqref{eqn:mKdVrel}.
In this section, we offer an alternate geometric method to define $\theta_n$ precisely at every vertex using the data of the continuous motion, and obtain the semi-discrete potential mKdV equation \eqref{mKdV} without the need for a difference equation.

To do this, consider the smooth curve frame $\tilde{F}_n(s)$ defined by
	\[
		\tilde{F}_n(s):=\left(\tilde{T}_n(s) \ \tilde{N}_n(s) \right), \quad\text{where }\tilde{T}_n(s):=x'_n(s) \text{ and }\tilde{N}_n(s):=R(\tfrac{\pi}{2})\tilde{T}_n(s).
	\]
Defining $\theta_n(s)$ as the \emph{tangential angle} of the smooth curve $x_n(s)$, i.e.\ $\theta_n(s)$ is the argument of $\tilde{T}_n(s)$, so that $\theta_n'(s)$ gives the smooth curvature, we have	
	\begin{equation}\label{eqn:theta1}
		\theta_n=\psi_n+w_n.
	\end{equation}
Together with \eqref{eqn:psi1} and \eqref{8}, one can verify
	\[
		\theta_{n+1}+\theta_n = \psi_{n+1}+\psi_n-\kappa_{n+1} \quad\text{and}\quad
		\theta_{n+1}-\theta_n = \kappa_{n+1}+w_{n+1}-w_n,
	\]
so that
	\begin{equation}\label{eqn:theta2}
		\theta_{n+1}=\psi_n-w_n.
	\end{equation}
Then it is straightforward to check using \eqref{eqn:theta1} and \eqref{eqn:theta2} that such $\theta_n$ satisfy all the relations in \eqref{eqn:theta}.

With such geometric definition of $\theta_n$, we now introduce an alternate method to directly obtain the semi-discrete potential mKdV equation.
By the definition of $w_n$, we have that $\tilde{F}_n=F_n R(w_n)$ so that by \eqref{7}, $\tilde{F}'_n$ satisfies
	\begin{equation}  \label{9}
		\tilde{F}'_n = F'_nR(w_n)+F_nR(w_n)' = F_n \left(M_n R(w_n)+R(w_n)' \right).
	\end{equation}  
On the other hand, by the well-known Frenet formula for plane curves, $\tilde{F}'_n$ can also be written as
	\begin{equation}\label{10}
		\tilde{F}'_n = F_n R(w_n)
					\begin{pmatrix}
						0 & -\theta'_n \\
						\theta'_n & 0  
					\end{pmatrix}.
	\end{equation}
Hence, from \eqref{9} and \eqref{10}, we have
	\[
		\psi'_n=(\theta_n-w_n)'=-\tfrac{2}{a_n}\sin{w_n},
	\]
allowing us to conclude that $\theta_n$ defined as the tangential angle satisfies the semi-discrete potential mKdV equation:
\begin{theorem}
	Let $x_n(s) : \Sigma \times I \to \mathbb{R}^2 \cong \mathbb{C}$ be a continuous motion of discrete curves as considered in  \cites{kaji_linkage_2019, inoguchi_explicit_2012}.
	The tangential angle $\theta_n(s)$ of the smooth curves in the continuous motion then satisfies the semi-discrete potential mKdV equation
	\begin{equation}\label{eqn:mKdV}
		\left(\frac{\theta_{n+1}+\theta_n}{2}\right)' = \frac{2}{a_n}\sin{\left(\frac{\theta_{n+1}-\theta_n}{2}\right)}.
	\end{equation}
\end{theorem}

In fact, \eqref{eqn:mKdV} tells us more:
\begin{proposition}\label{prop:isoDarboux}
	Let $x_n(s) : \Sigma \times I \to \mathbb{R}^2 \cong \mathbb{C}$ be a continuous motion of a discrete curve $x^0_n = x_n(s_0)$ for some $s_0 \in I$ resulting in the semi-discrete potential mKdV equation (as considered in \cites{kaji_linkage_2019, inoguchi_explicit_2012}).
	Giving $x^0$ the arc-length polarization so that $\mu_{(n, n+1)} = \frac{1}{a_n^2}$, $x_n(s)$ is an infinitesimal Darboux transformation keeping the arc-length polarization.
\end{proposition}
\begin{proof}
	Note that on any edge $(n, n+1)$,
		\[
			\frac{x'_n x'_{n+1}}{(x_n - x_{n+1})^2} = \frac{1}{a_n^2}e^{\sqrt{-1} (\theta_n + \theta_{n+1} - 2 \psi_n)} = \frac{1}{a_n^2} = \frac{\mu_{(n, n+1)}}{m},
		\]
	where $m \equiv 1$.
	Since $s$ is the arc-length parameter of $x_n(s)$ for any $n \in \Sigma$, Remark~\ref{rem:arclength} implies that $\frac{\dif s^2}{m}$ is the (smooth) arc-length polarization of $x_n$, and hence Proposition~\ref{prop:discArc} gives us the desired conclusion.
\end{proof}

Now assume that $x^0 : (\Sigma, \frac{1}{\mu}) \to \mathbb{C}$ is a discrete polarized curve with arc-length polarization $\mu$, and let $x$ be an infinitesimal Darboux transformation of $x^0 = x(0)$ with parameter function $m$, keeping the discrete arc-length polarization so that
	\[
		\frac{1}{\mu_{(n,n+1)}} = |x_n(s) - x_{n+1}(s)|^2
	\]
for any $s \in I$ giving us the isoperimetric condition.
Then Proposition \ref{prop:discArc} tells us that $\frac{\dif s^2}{m}$ is the (smooth) arc-length polarization of the curve $x_{n_0}(s)$ for any $n_0 \in \Sigma$.
Reparametrizing so that $m \equiv 1$, then $s$ is the arc-length parametrization for $x_{n_0}(s)$ for any $n_0 \in \Sigma$, telling us that we obtain the continuous motion of discrete curves resulting in semi-discrete potential mKdV equation as considered in \cites{kaji_linkage_2019, inoguchi_explicit_2012}.

Summarizing, we have:
\begin{theorem}\label{thm:semidiscrete}
	The following transformations or motions describe the same semi-discrete system:
		\begin{itemize}
			\item (smooth) arc-length polarization preserving (iterated) Darboux transformations of a polarized curve,
			\item (discrete) arc-length polarization preserving infinitesimal Darboux transformation of a discrete polarized curve, and
			\item certain isoperimetric motion of a discrete curve resulting in semi-discrete potential mKdV equation.
		\end{itemize}
\end{theorem}

\subsection{The semi-discrete potential mKdV equation via infinitesimal Darboux transformation}
Theorem \ref{thm:semidiscrete} allows us to introduce an efficient route to obtaining the semi-discrete potential mKdV equation without any use of frames via infinitesimal Darboux transformation keeping arc-length polarization.
Let $x : \Sigma \times I \to \mathbb{C}$ be an infinitesimal Darboux transformation keeping the discrete arc-length polarization of a discrete curve $x^0 : (\Sigma,\frac{1}{\mu}) \to \mathbb{C}$.
Then, we can see that on any edge $(n, n+1)$, $x_n$ and $x_{n+1}$ are an arc-length polarized Darboux pair with parameter $\mu_{(n, n+1)}$.
Assume without loss of generality that $s \in I$ is the simultaneous arc-length parameter of $x_n(s)$ for all $n \in \Sigma$ so that $m \equiv 1$.

For any $n \in \Sigma$, define $k_n$ to be the curvature of $x_n$ so that $x_n''=i k_n x_n'$, and let $\theta_n$ be the tangential angle of $x_n$, i.e.\ $x_n' = e^{i\theta_n}$, so that $\theta'_n = k_n$.
Now on any edge $(n, n+1)$, we compute $(x'_n x'_{n+1})'$ in two ways.
Writing $\mu = \mu_{(n, n+1)}$ for simplicity, first, by the tangential cross-ratio condition,
	\[
		(x_n' x_{n+1}')' = \left( \mu (x_{n+1}-x_n)^2 \right)' 
                  	= 2\sqrt{\mu} (e^{i \theta_{n+1}}-e^{i \theta_n})
                                      e^{i \tfrac{\theta_n\vphantom{\theta_{n+1}}}{2}} e^{i \tfrac{\theta_{n+1}}{2}}.
\]
On the other hand,
	\[
		(x_n' x_{n+1}')'
			= i \theta_n' x_n' x_{n+1}'+i \theta_{n+1}' x_n' x_{n+1}'
			= i(\theta_{n+1}'+\theta_n')e^{i\theta_n} e^{i\theta_{n+1}}.
	\]
Hence, using the fact that $|x_{n+1} - x_n |^2 = \tfrac{1}{\mu}$, we obtain the semi-discrete potential mKdV equation \eqref{mKdV}:
	\begin{align*}
		\left(\frac{\theta_{n+1}+\theta_n}{2}\right)' & = -\sqrt{\mu} \, i \, (e^{i\theta_{n+1}}-e^{i\theta_n})
                                          e^{-i\tfrac{\theta_n\vphantom{\theta_{n+1}}}{2}} e^{-i\tfrac{\theta_{n+1}}{2}} \nonumber \\
                                & = \frac{2}{| x_{n+1}-x_n |} \sin{\left(
                                                                              \frac{\theta_{n+1}-\theta_n}{2}
                                                                             \right)}.
	\end{align*}

\vspace{15pt}
\textbf{Acknowledgements.}
The authors would like to thank Professor Francis Burstall, Professor Udo Hertrich-Jeromin and the referees for their valuable comments, and gratefully acknowledges the support from JSPS/FWF Bilateral Joint Project I3809-N32 ``Geometric shape generation" and JSPS Grant-in-Aids for: JSPS Fellows 19J10679, Scientific Research (C) 15K04845, (C) 20K03585 and (S) 17H06127 (P.I.: M.-H.\ Saito).

\end{document}